\documentclass[12]{amsart}

\usepackage{amssymb}
\usepackage{graphicx}
\usepackage{psfrag}
\usepackage{cite}

\newtheorem{theorem}{Theorem}[subsection]
\newtheorem{lemma}[theorem]{Lemma}
\newtheorem{proposition}[theorem]{Proposition}
\newtheorem{corollary}[theorem]{Corollary}
\newtheorem*{Theorem}{Theorem}

\theoremstyle{definition}
\newtheorem{definition}[theorem]{Definition}

\theoremstyle{remark}
\newtheorem{remark}[theorem]{Remark}

\newcommand{\link}{\operatorname{Lk}}

\renewcommand{\star}{\operatorname{St}}

\newcommand{\bbZ}{\mathbb Z}
\newcommand{\bbQ}{\mathbb Q}
\newcommand{\bbH}{\mathbb H}

\newcommand{\bt}{\mathbf t}
\newcommand{\bd}{\mathbf d}

\newcommand{\cA}{\mathcal A}
\newcommand{\cB}{\mathcal B}
\newcommand{\cE}{\mathcal E}

\newcommand{\cL}{\mathcal L}

\newcommand{\cS}{\mathcal S}
\newcommand{\cG}{\mathcal G}
\newcommand{\cH}{\mathcal H}

\renewcommand{\bar}[1]{\overline{#1}}
\renewcommand{\tilde}[1]{\widetilde{#1}}

\begin{document}

\title{Eulerian cube complexes and reciprocity}

\author{Richard Scott}
\address{Department of Mathematics and Computer Science\\
Santa Clara University\\
Santa Clara, CA  95053}
\email{rscott@scu.edu}

\date{July 2013, revised February 2014}

% \keywords{Coxeter groups, growth series, automatic groups}

\begin{abstract}
Let $G$ be the fundamental group of a compact nonpositively curved
cube complex $Y$.  With respect to a basepoint $x$, one obtains an
integer-valued length function on $G$ by counting the number of edges
in a minimal length edge-path representing each group element.  The
{\em growth series of $G$ with respect to $x$} is then defined to be
the power series $G_x(t)=\sum_g t^{\ell(g)}$ where 
$\ell(g)$ denotes the length of $g$.  Using the fact that $G$ admits a
suitable automatic structure, $G_x(t)$ can be shown to be a rational
function.  We prove that if $Y$ is a manifold of dimension $n$, then
this rational function satisfies the reciprocity formula
$G_x(t^{-1})=(-1)^n G_x(t)$.  We prove the formula in a more general
setting, replacing the group with  the fundamental groupoid, replacing
the growth series with the characteristic series for a suitable
regular language, and only assuming $Y$ is Eulerian. 
\end{abstract}

\maketitle

\section{Introduction}

Given a group $G$ and a length function $\ell: G\mapsto\bbZ$, the {\em
  growth series} of $G$ with respect to $\ell$ is the formal power series 
\[G_{\ell}(t)=\sum_{g\in G}t^{\ell(g)}.\]
There are several instances in the literature of known reciprocity
formulas for growth series when the length function is derived from
some action of the group on a manifold or nonsingular cell complex.
In \cite{Serre}, Serre observed that for affine Coxeter groups and the
standard word length, $G(t^{-1})=\pm G(t)$.  Other examples for
Fuchsian groups were noted by Floyd and Plotnick \cite{FP}.
The formulas for affine Coxeter groups were generalized by
Charney and Davis \cite{CD} to multivariate growth series for Coxeter
groups whose associated Davis complex was an Eulerian complex, and a
geometric interpretation was later given by Dymara 
\cite{Dym} and Davis et.al. \cite{DDJO} in terms of ``weighted
$L^2$-cohomology'' and Poincar\'{e} duality.  A multivariate
noncommuative version of the formula was described by the
author \cite{Scott} for right-angled Coxeter groups with Eulerian
Davis complex, and a partial interpretation in terms of weighted
$L^2$-cohomology was given in work with Okun \cite{OS}.

Right-angled Coxeter groups fall into the more general class of groups
acting on CAT($0$) cube complexes.  This paper arose from the
observation that an extension of the author's methods in
\cite{Scott} establishes the reciprocity formula for many more groups in
this class.  In particular, reciprocity also holds for the fundamental
group of any compact, Eulerian, nonpositively curved cubical complex.
We stick to this case for the remainder of the introduction
(and most of the paper) since the argument is easier to follow.  A
generalization to orbihedra (which includes the case of right-angled
Coxeter groups) is mentioned in Section~\ref{ss:orb} at the end of the
paper.  

Let $X$ be a connected CAT($0$) cube complex and let $G$ be a group
acting freely, cellularly, and cocompactly on $X$.  Then the quotient
space $Y=X/G$ is a nonpositively curved cube complex with universal
cover $X$ and fundamental group isomorphic to $G$.  
We let $V$ denote the vertices of $X$, and $V/G$ denote the
vertices of $Y$.  Then the set of homotopy classes of paths in $Y$
that start and end at vertices in $V/G$ forms a groupoid which we
denote by $\cG$.  We let $\cG_{x,y}$ denote the morphisms in $\cG$
that start at $x$ and end at $y$.  The vertex group $G_x=\cG_{x,x}$
is precisely the fundmental group $\pi_1(Y,x)$ and, hence, isomorphic
to $G$.

The complex $X$ comes equipped with a family of {\em hyperplanes,}
obtained by extending the perpendicular bisectors of the
$1$-dimensional cubes.  We denote the set of hyperplanes in $X$ by
$\cH$, and their images in $Y$ by $\cH/G$.    For each element of
$\cG$, we define a multivariable ``length'' by counting the number of
times a representative path crosses each hyperplane in $\cH/G$.  More
precisely, we let $\bbQ((\bt))$ denote the multivariate ring of formal
Laurent series with 
indeterminates $\bt=(t_h)$ indexed by elements $h\in\cH/G$.  For any
homotopy class $[\gamma]\in\cG_{x,y}$, we choose a representative
$\gamma$ that lies in the $1$-skeleton of $Y$ (i.e., an ``edge path''
in $Y$) and we assume that it is a minimal length representative (uses
as few edges as possible).  We then define the weight of $[\gamma]$ to
be  
\[\tau([\gamma])=\prod_{h\in\cH/G}t_h^{m_h}\]
where $m_h$ counts the number of times that the edge path $\gamma$
crosses the hyperplane $h$.  It can be shown
(Section~\ref{s:growth-series}) that the weight on $[\gamma]$ is
independent of the choice of representative, so we obtain a
well-defined {\em (multivariate) growth series}
$\cG_{x,y}(\bt)\in\bbQ((\bt))$ by summing over $\cG_{x,y}$:
 \[\cG_{x,y}(\bt)=\sum_{[\gamma]\in\cG_{x,y}}\tau([\gamma]).\]
The {\em (ordinary) growth series} is the single-variable power series
$\cG_{x,y}(t)$ obtained by substituting $t$ for each indeterminate $t_{h}$ in
the multivariate growth series.  Using a result of Niblo and Reeves
\cite{NR}, it can be shown that these growth series are, in fact,
rational functions.

In order to state our theorem, we briefly recall the definition
of an Eulerian complex.  Let $K$ be a cell complex such that the link
of every (nonempty) cell is a simplicial complex.  We say that $K$ is
{\em Eulerian of dimension $n$} if every maximal cell is
$n$-dimensional and for every (nonempty) cell $\sigma$ in $K$, the
Euler characteristic of the link is equal to the Euler characteristic
of the sphere of the same dimension, i.e., 
\[\chi(\link(\sigma))=1+(-1)^{\dim\link(\sigma)}.\]
If, in addition, the Euler characteristic of $K$ is the same as the
Euler characteristic of the $n$-sphere, then $K$ is called an {\em Eulerian
$n$-sphere}.  In particular, if $K$ is an $n$-dimensional homology
manifold, then it is Eulerian, and if in addition $n$ is odd, $K$ is an Eulerian
$n$ sphere (by Poincar\'{e} duality).  The main result of this paper is
the following theorem.   

\begin{Theorem} Let $G$ be a group acting freely,
  cellularly, and cocompactly on a connected CAT($0$) cube complex $X$, and let
  $V$ be the set of vertices of $X$.  If $X/G$ is Eulerian of
  dimension $n$, then for any $x,y\in V/G$, the multivariate growth
  series $\cG_{x,y}(\bt)$ (regarded as a rational function) satisfies 
\[\cG_{x,y}(\bt^{-1})=(-1)^n\cG_{x,y}(\bt).\]
\end{Theorem}

The proof of the theorem actually establishes a more general formula
that holds for a certain regular language encoding the fundamental
groupoid. In Section~\ref{s:languages}, we collect and discuss
requisite facts
from formal language theory. In Section~\ref{s:cc}, we discuss groups acting on
CAT($0$) cube complexes.  In Section~\ref{s:growth-series}, we discuss growth series and
give some sample computations.  In Section~\ref{s:pf}, we give the proof of the
main theorem. 

\section{Regular languages}\label{s:languages}

Given a set $\cA$, we let $\cA^*$ denote the free monoid on $\cA$.  We
refer to $\cA$ as an {\em alphabet} and elements of $\cA^*$ as {\em
  words} over $\cA$.  Any subset $\cL\subseteq\cA^*$ is called a {\em
language} over $\cA$.  A language is called {\em regular} if it is the
language accepted by some finite state automaton.  For our purposes, a
finite state automaton will consist of a finite directed graph with vertex set
$\cS$ (the {\em state set}) and two designated subsets $\cB,\cE\subseteq
\cS$ (the {\em initial states} and the {\em accept states}, respectively).  The
directed edges (called {\em transitions}) of the graph are labeled by
elements of some alphabet $\cA$, and the labeling is further assumed
to have the property that for each $a\in\cA$ and each $i\in\cS$ there
is at most one edge labeled $a$ emanating from $i$.  Any directed path
in the graph then determines a word in $\cA^*$ by writing down (left
to right) the labels on the consecutive edges of the path. The
language {\em accepted} by the automaton is the set of all words
corresponding to paths that start at one of the initial or ``begin''
states $i\in\cB$ and end at one of the accept or ``end'' states
$i\in\cE$.  

\begin{remark} Our definition of a finite state automaton is less
  restrictive than a {\em deterministic} one (we allow more than one
  start state), but more restrictive than a {\em nondeterministic} one
  (we don't include a padding symbol).  Since deterministic and
  nondeterministic automata define the same language class (i.e.,
  regular languages), so does ours. 
\end{remark}    

\subsection{Characteristic series}\label{ss:char-series}  
Formal languages are often considered in an algebraic context by
identifying them with their characteristic series, a power series
whose terms are the words in the language.  For convenience, we shall
work in the ring $\bbQ((\cA))$ consisting of formal Laurent series
over $\bbQ$ in noncommuting indeterminates indexed by $\cA$.  To avoid
extra notation, we use the same symbols for these indeterminates as
for the elements of $\cA$, and for $a\in\cA$, we let
$a^{-1}$ denote the formal inverse in $\bbQ((\cA))$.  Given a
language $\cL$ over $\cA$, we define its {\em  characteristic series}
to be  
\[\lambda=\sum_{\alpha\in\cL}\alpha,\]
which we regard as an element of $\bbQ((\cA))$.  If $\cL$ is a regular
language, then its characteristic series $\lambda$ has the
following rational algebraic representation (see, e.g.,
\cite[Theorem~5.1]{SS}): 
\begin{equation}\label{eq:char-series}
\lambda=B(I-Q)^{-1}E=B(I+Q+Q^2+Q^3+\cdots)E
\end{equation}
where $B$ is a row vector with entries in $\bbQ$, $E$ is a column
vector with entries in $\bbQ$, and $Q$ is a square matrix whose
entries are formal sums of elements in $\cA$.  In fact, given a finite 
state automaton that accepts $\cL$, the matrices $B$, $E$, and $Q$ can
be given explicitly as follows.
\begin{itemize}
\item $Q$ is the $\cS\times\cS$ matrix 
whose $(i,j)$-entry is the sum of all $a\in\cA$ that label transitions
from $i$ to $j$.
\item $B$ is the $1\times\cS$ matrix with all zeros except for $1$s in
  the entries corresponding to initial states.
\item $E$ is the $\cS\times 1$ matrix with all zeros
except for $1$s in the entries corresponding to accept states.  
\end{itemize}
Since $Q$ is determined by the transitions in the automaton, we shall
refer to $Q$ as a {\em transition matrix} for $\lambda$.

\subsection{Reciprocity}  
If $\cL$ is a regular language, then the entries of
any transition matrix $Q$ will be a sum of elements in $\cA$.  By
replacing each such element with its reciprocal (i.e., its formal
inverse), we obtain an $\cS\times\cS$ matrix $\bar{Q}$ defined over
$\bbQ((\cA))$.  We then define the {\em reciprocal} of the
characteristic series $\lambda$ by
\[\bar{\lambda}=B(I-\bar{Q})^{-1}E,\]
provided that $I-\bar{Q}$ is invertible over $\bbQ((\cA))$  (note that
the formal expansion $I+\bar{Q}+\bar{Q}^2+\bar{Q}^3+\cdots$ is not
defined when $\bar{Q}$ has terms with negative exponents, so
invertibility of $I-\bar{Q}$ is not automatic).  It can be shown
that if the reciprocal of the characteristic series for a regular
language exists, then it is independent of the choice of automaton
(\cite[Prop.~4.1]{Scott}).  

\subsection{Specialization}\label{ss:spec}  
Given any ring homomorphism
$\phi:\bbQ((\cA))\rightarrow R$, we refer to the image of any element 
$\lambda\in\bbQ((\cA))$ as the {\em specialization of $\lambda$ (with
  respect to $\phi$)}.  The most common examples are 
specializations to commutative Laurent series rings induced by
monomial substitutions.  More precisely, given an arbitrary indexing
set $I$, we let $\bbQ_I((\bt))$ (or 
simply $\bbQ((\bt))$ if $I$ is clear from the context) denote the ring
of Laurent series in commuting indeterminates $\bt=(t_i)_{i\in I}$.
Then any assignment $a\mapsto \bt_a$ from elements in $\cA$ to
nontrivial monomials in $\bbQ((\bt))$ induces a specialization
homomorphism $\phi:\bbQ((\cA))\rightarrow\bbQ((\bt))$.  For such a specialization,
we denote the image of $\lambda\in\bbQ((\cA))$ by $\lambda(\bt)$,
suppressing the index set $I$ and the homomorphism $\phi$.  If
$\lambda$ is the characterstic series for a regular language $\cL$,
then the series $\lambda(\bt)$ is 
the power series expansion of a rational function. (This follows from
the fact that the specialization of $I-Q$ in the representation
(\ref{eq:char-series}) is invertible
over the ring of rational functions in $\bt$.)  For such a
characteristic series $\lambda$, if the reciprocal $\bar{\lambda}$
exists, then its specialization is given by
$\bar{\lambda}(\bt)=\lambda(\bt^{-1})$, where $\bt^{-1}$ denotes the
tuple $(t_i^{-1})_{i\in I}$.     

\section{Cube complexes and cube paths}\label{s:cc}

A {\em cube complex} is a piecewise-Euclidean metric cell complex
obtained by gluing Euclidean cubes together via isometries along their
faces.  A cube complex is {\em CAT($0$)} if it is simply-connected,
and the link of every vertex is a flag complex (a {\em flag} complex is
a simplicial complex such that every set of pairwise-adjacent
vertices spans a simplex).  If $X$ is CAT($0$) and $v$ is a vertex, we
let $\link(v)$ denote the link of $v$.   

\subsection{Diagonals and cube paths}
Let $X$ be a CAT($0$) cube complex.  Every cell in $X$ is an
isometrically embedded cube.   A segment that starts at one
vertex of an embedded cube and ends at the
opposite vertex will be called a {\em (directed) diagonal} in $X$.
Note that we include the set of {\em trivial diagonals}, which start
and end at the same vertex.  For a diagonal $d$, we adopt the
following notation: 
\begin{itemize}
\item $d^*$ is the oppositely oriented diagonal ($d^*=d$ iff $d$ is a
  trivial diagonal),
\item $\alpha(d)$ is the initial vertex,
\item $\omega(d)$ is the terminal vertex,
\item $C(d)$ is the cube spanned by $d$,
\item $\dim d$ or $|d|$ is the dimension of the cube $C(d)$,
\item $\sigma(d)$ is the image of $C(d)$ in $\link(\alpha(d))$,
\end{itemize} 
We define a {\em cube path} in $X$ to be a sequence
$\bd=(d_1,\ldots,d_n)$ of diagonals such that
$\omega(d_i)=\alpha(d_{i+1})$ for $i=1,\ldots,n-1$.  Note that this 
condition means that for each $i$, $\sigma(d_{i}^*)$ and $\sigma(d_{i+1})$ are both
simplices in the link of the vertex $v_i=\omega(d_i)$.  The {\em
  length} of a cube path $\bd$, which we denote by $|\bd|$ is the sum
of the dimensions of the corresponding cubes, i.e., 
\[|\bd|=|d_1|+\cdots+|d_n|.\]
A cube path is {\em reduced} if it contains no trivial diagonals.  
A reduced cube path is called {\em normal} if
$\star(\sigma(d_i^*))\cap\sigma(d_{i+1})=\emptyset$. Here
$\star(\sigma)$, the {\em star} of $\sigma$, denotes the
union of all simplices in the link that contain $\sigma$ as a face.

\begin{remark} The defining condition for normal cube paths are often
  given in terms of cubes in $X$ rather than simplices in the link.
  In this case, the requirement would be 
  $\star(C(d_i))\cap C(d_{i+1})=\{v_i\}$ (where $\star(C)$ now denotes the
  union of all {\em cubes} in $X$ that contain $C$ as a face).
\end{remark}

\begin{proposition}[Proposition 3.3 in \cite{NR}]\label{prop:unique-greedy}  
If $X$ is a connected CAT($0$) cube complex and $u$ and $v$ are any two
  vertices in $X$, then there exists a
  unique normal cube path from $u$ to
  $v$.  
\end{proposition}

\subsection{The fundamental groupoid induced by a $G$-action}

Let $G$ be a group acting freely, cellularly, and cocompactly on a
connected CAT($0$) cube complex $X$.  Then $X$ is the universal cover
of $X/G$ and we let $p:X\rightarrow X/G$ denote the projection map.
Let $V$ denote the vertex set of $X$.  Since $G$ acts on $V$, we can
define a groupoid $\cG$ whose objects are the orbits $V/G$ and whose
morphisms are homotopy classes of paths in $X/G$ that start and end in
$V/G$.  In other words, $\cG$ is the subgroupoid of the fundamental
groupoid $\pi_1(X/G)$ obtained by restricting the objects to the
subset $V/G\subseteq X/G$. Given  $x,y\in V/G$, we let $\cG_{x,y}$
denote the set of morphisms from $x$ to $y$, and we let $\cG_x$ denote
the vertex group $\cG_{x,x}$.  Since the action on $X$ is free,
$\cG_x$ coincides with the fundamental group $\pi_1(X/G,x)$, which is
in turn isomorphic to $G$. 

Given a cube $C$ or diagonal $d$ in $X$, we let $\bar{C}$ and 
$\bar{d}$ denote the respective projections in $X/G$.  We call
$\bar{C}$ (resp., $\bar{d}$) a {\em cube in $X/G$} (resp., {\em
  diagonal in $X/G$}).  Any diagonal $\bar{d}$ has a well-defined
initial and final vertex in $V/G$ defined by
$\alpha(\bar{d})=\bar{\alpha(d)}$ and
$\omega(\bar{d})=\bar{\omega(d)}$.  Thus, we can define a {\em cube
  path in $X/G$} to be any sequence of diagonals
$\bar{\bd}=(\bar{d}_1,\bar{d}_2,\ldots,\bar{d}_n)$ satisfying
$\omega(\bar{d}_i)=\alpha(\bar{d}_{i+1})$.   Given a lift $v$ for the
initial vertex $\alpha(\bar{d}_1)$, any such path has a unique lift to
a cube path in $X$ starting at $v$.  Thus, cube paths in $X/G$
correspond to $G$-orbits of cube paths in $X$.   A cube path in $X/G$
will be called {\em normal} if any (hence every) lift is a normal
cube path in $X$.

Given $x,y\in V/G$ and respective lifts $\tilde{x},\tilde{y}\in X$,
any (continuous) path $\gamma$ from $x$ to $y$ in $X/G$ has a unique lift
to a path $\tilde{\gamma}$ from $\tilde{x}$ to $g\tilde{y}$ for some 
(unique) $g\in G$.  Since $X$ is the universal cover for $X/G$, any path
homotopic to $\tilde{\gamma}$ in $X$ projects to a path homotopic to
$\gamma$ in $X/G$.  Since any path in $X$ with endpoints in $V$ is
homotopic to a cube path, any element of the groupoid $\cG$ is
represented by a cube path in $X/G$.  Moreover, since there is a
unique normal representative for any such cube path (by
Proposition~\ref{prop:unique-greedy}), the elements (morphisms) of
$\cG$ correspond bijectively to normal cube paths in $X/G$ or,
equivalently, to $G$-orbits of normal cube paths in $X$.

\subsection{Automata for the groupoid}
Let $\cA$ denote the set of all nontrivial diagonals in
$X/G$.  Given any cube path in $X/G$, we obtain a word in $\cA^*$ by
reading off the nontrivial diagonals in the cube path.
We let $\cL\subseteq\cA^*$ denote the set of words corresponding to
normal cube paths.  By the previous paragraph, we
have a bijections $\cL\rightarrow \cG$, hence
$\cL$ defines a normal form for the groupoid $\cG$.   

\begin{proposition}[Propositions 5.1 and 5.2 of \cite{NR}]
The normal form $\cL$ provides a biautomatic structure  for $\cG$ (in
the sense of \cite[Chapter~11]{epstein}).  In particular, $\cL$ is a
regular language. 
\end{proposition}

Given $x,y\in V/G$, we let $\cL_{x,y}\subseteq \cL$ denote the
sublanguages corresponding to $\cG_{x,y}$.  A non-deterministic
automaton for $\cL$ is given in \cite{NR} using $\cA$ as both the
alphabet and the state set.  Since we will be interested in the
characteristic series for the sublanguages $\cL_{x,y}$, we modify the
automaton in \cite{NR} by enlarging the state set to $\cS=\cA\cup
(V/G)$.  Thus, elements of $\cS$ are precisely the diagonals (both
nontrivial and trivial) in $X/G$.  Before defining the transitions, we
first note that because the $G$-action is free, we  have the following:
\begin{itemize}
\item the map $j\mapsto j^*:\cS\rightarrow\cS$ given by
  $j^*=\bar{d^*}$ where $d$ is any lift of $j$ is a well-defined
  involution on $\cS$, and 
\item for any vertex $v\in V$, the projection
  $\sigma\mapsto\bar{\sigma}$ defines an isomorphism from the link of
  $v$ in $X$ to the link of $\bar{v}$ in $X/G$.  For any diagonal $d$
  with initial vertex $v$, we let $\sigma(\bar{d})$ denote the image
  of $\sigma(d)$ in $\link(\bar{v})$.
\end{itemize}
The transitions for our automaton are then defined as follows.  Given
  (nontrivial) states 
$i,j\in\cA$, there is a transition from $i$ to 
$j$ labeled by $i$ whenever $\omega(i)=\alpha(j)$ and
$\star(\sigma(i^*))\cap\sigma(j)=\emptyset$.  We also have transitions
from states in $\cA$ to states in $V/G$.  Namely, for each $i\in\cA$,
we add a transition from $i$ to $y\in V/G$ labeled $i$ whenever
$\omega(i)=y$.  Since the condition defining transitions is precisely
the condition defining normal cube paths, we obtain an automaton that
accepts $\cL_{x,y}$ by taking the initial states to be
$\cB_{\alpha}=\{i\;|\; \alpha(i)=x\}$ and the accept states to be the
singleton set $\cE=\{y\}$. 

\begin{remark}\label{rem:alt-auto}
The automaton above has transitions corresponding to pairs of composable
diagonals such that the final vertex of the first diagonal is the
initial vertex of the second.  One obtains a different automaton
(accepting the same language) by defining transitions for composable
diagonals when the initial vertex of the first diagonal is the
final vertex of the second.  More precisely, given $i,j\in\cS$, there
is a transition from $i$ to $j$ labeled $i^*$ whenever
$\alpha(i)=\omega(j)$ and $\star(\sigma(i))\cap\sigma(j^*)=\emptyset$.
Since accept states will now correspond to initial vertices of
diagonals, we have additional transitions from $i\in\cA$
to $y\in V/G$ labeled $i^*$ whenever $\alpha(i)=y$.  The language $\cL_{x,y}$ is
then accepted by the automaton with initial states
$\cB_{\omega}=\{i\;|\;\omega(i)=x\}$, and with accept states the singleton
$\cE=\{y\}$.
\end{remark}

\begin{remark}
The nondeterministic automaton for the full language $\cL$ described
in \cite{NR} has states corresponding only to the non-vertex diagonals
$\cA$, and every state is both an initial state and an accept state.
The transitions (directed edges in the automaton) are the same as
ours, but are labeled by the final state rather than the initial
state.  That is, the transition from $i$ to $j$ is labeled by $j$
rather than by $i$.  There is no {\em a 
  priori} reason to prefer one convention over the other; we have
chosen ours different from \cite{NR} only because it matches the
setup in \cite{Scott}, and our proofs are based on those results.
On the other hand, the fact that our automaton uses more states than
just those in $\cA$ is worth justifying.  The key issue is that with
only nontrivial diagonals as states, the map from paths in the
automaton to words in $\cL$ would not be injective.  For example, if
the transitions are labeled by their final states, as in \cite{NR}, then 
{\em every} edge in the automaton that ends at the state $j$ will
correspond to the same word in $\cL$, namely $j$.  Likewise, if
transitions are labeled by their initial states, then every edge in
the automaton that starts at $i$ will correspond to the same word $i$.
Adding states corresponding to vertices in $X/G$, with the appropriate
transitions, fixes the non-injectivity problem. 
\end{remark}

\subsection{Characteristic series for the groupoid}\label{ss:gcs}

Let $\lambda_{x,y}$ denote the characteristic series for
$\cL_{x,y}$.  Then by discussion in \ref{ss:char-series} and the
definition of the automaton above, we have the rational
representation:  
\begin{equation}\label{eq:lambda+}
\lambda_{x,y} = B_{\alpha}(I-Q_+)^{-1}E 
\end{equation}
where $B_{\alpha}$ is the $1\times\cS$ row vector with all zeros except $1$s
in the entries corresponding to the diagonals with initial vertex $x$,
$E$ is the $\cS\times 1$ column vector with all zeros except for a
$1$ in the entry corresponding to $y$, and
$Q_+$ is the $\cS\times\cS$ matrix given by  
\[(Q_+)_{i,j}=\left\{\begin{array}{ll}
i & \mbox{if $i,j\in\cA$, $\alpha(j)=\omega(i)$, and
  $\star(\sigma(i^*))\cap\sigma(j)=\emptyset$}\\
i & \mbox{if $i\in\cA$, $j\in V/G$, and $\omega(i)=j$}\\
0 & \mbox{otherwise}\end{array}\right.\]

Alternatively, using the other automaton described in
Remark~\ref{rem:alt-auto}, we have  
\begin{equation}\label{eq:lambda-}
\lambda_{x,y} = B_{\omega}(I-Q_-)^{-1}E 
\end{equation}
where $B_{\omega}$ is the $1\times\cS$ row vector with all zeros except $1$s
in the entries corresponding to the diagonals with {\em final} vertex $x$,
$E$ is as before, and $Q_-$ is the $\cS\times\cS$ 
matrix given by  
\[(Q_-)_{i,j}=\left\{\begin{array}{ll}
i^* & \mbox{if $i,j\in\cA$, $\alpha(i)=\omega(j)$, and
  $\star(\sigma(i))\cap\sigma(j^*)=\emptyset$}\\
i^* & \mbox{if $i\in\cA$, $j\in V/G$, and $\alpha(i)=j$}\\
0 & \mbox{otherwise}\end{array}\right.\]

The transition matrices $Q_+$ and $Q_-$ can be related using the
involution $j\mapsto j^*:\cA\rightarrow\cA$.  For convenience, we
extend this involution $\cS$ by having it act trivially on the
trivial diagonals $V/G$.  The following fact follows from the explicit
descriptions given above for the transition matrices.  

\begin{proposition}\label{prop:pmQ}
Let $[\ast]$ denote the $\cS\times\cS$ permutation matrix induced by
the involution $j\mapsto j^*:\cS\rightarrow\cS$.  Then 
\[Q_-=[*]Q_+[*].\]
\end{proposition}

\section{Growth series and examples}\label{s:growth-series}

Throughout this section $X$ is a CAT($0$) cube complex, and $G$ is
a group acting freely, cellularly, and cocompactly on $X$.  
 
\subsection{Hyperplanes and weights on cube paths}
Parallelism between (unoriented) edges in each cube extends to an  
equivalence relation on the edges of $X$ as follows.  Two 
edges $e$ and $e'$ are equivalent if there exists a sequence of edges
$e=e_1,e_2,\ldots,e_n=e'$ and a sequence of cubes $C_1,\ldots,C_{n-1}$ such
that $e_i$ is parallel to $e_{i+1}$ in the cube $C_i$.  Given an edge
  $e$, we let $[e]$ denote its equivalence class in $X$, and we let $\cH$
denote the set of all such equivalence classes.  

\begin{remark}
Given a cube $C$ in $X$, we define a {\em midplane} to be the intersection of $C$
with any hyperplane passing through the geometric center of $C$ which
is parallel to a codimension-one face.   Any midplane bisects the edges which are
perpendicular to it.  Given a parallel class of edges in $X$, we
define the corresponding {\em hyperplane} or {\em wall} to be the
union of all midplanes that bisect some edge in the parallel class.
Any such hyperplane is an isometrically embedded CAT($0$) cube complex
and separates $X$ into two parts called {\em halfspaces}.  There is a
bijection from parallel classes of edges to the set of hyperplanes
given by mapping the parallel class of an edge $e$ to the (unique)
hyperplane spanned my any midplane that bisects $e$. (Details can be
found in \cite{Sageev}.)  For this reason, we shall often refer to the
elements of $\cH$ as hyperplanes, and will say that an edge $e$ {\em
  crosses} the hyperplane $H$ if $[e]=H$.   
\end{remark}

Now let $\cH/G$ denote the set of $G$-orbits of hyperplanes, and let
$\bbQ((\bt))$ denote the ring of Laurent series in commuting
indeterminates $\bt=(t_h)$ indexed by $h\in\cH/G$.  We then obtain a
monomial ``weighting'' on the set of cube paths in $X/G$ as follows.
First we define this weighting on edge paths.  By definition, an {\em
  edge path} is a cube path consisting only of $1$-dimensional
diagonals (i.e., consisting only of {\em oriented} edges).  For an
edge path $\bd=(d_1,\ldots,d_n)$ in $X$, we let $[d_1],\ldots,[d_n]$
denote the corresponding sequence of parallel classes (forgetting
orientations), and we let $h_1,\ldots,h_n$ denote the corresponding
$G$-orbits in $\cH/G$.  Then we define the {\em   $G$-weight} of $\bd$
to be the monomial in $\bbQ((\bt))$ given by      
\[\tau(\bd)=t_{h_1}t_{h_2}\cdots
t_{h_n}.\]
In other words, if we let $m_h$ denote the number of times the edge
path $\bd$ crosses a hyperplane in the orbit $h$, then 
\[\tau(\bd)=\prod_{h}t_h^{m_h}.\]

More generally, we define the $G$-weight of {\em any} cube path.
Given a diagonal $d$ with $\dim d\geq 1$, we replace it with a minimum length
edge path that starts at $\alpha(d)$ and ends at $\omega(d)$.  If
$\dim d=m$ then this edge path will cross $m$ hyperplaners, and this
set of hyperplanes depends only on $d$, not the choice of edge path.
Given an arbitrary cube path $\bd$, we then 
replace each diagonal $d_i$ with a corresponding edge path and define
$\tau(\bd)$ to be the $G$-weight of the resulting edge path.   (Note
that the length of a cube path $\bd$ is precisely the length of a
representative edge path; this was the reason for defining the length
$|\bd|$ as we did.)  

\begin{proposition}[Sageev \cite{Sageev}]\label{prop:wwd} 
  A minimum length edge path crosses a hyperplane at
  most once, and two minimum length edge paths with the same endpoints
  must must cross the same set of hyperplanes.   
\end{proposition}

In particular, the $G$-weight of a minimum length cube path depends
  only on its endpoints.

\subsection{Growth series}
We can transfer the $G$-weighting on cube paths to a weighting on the
groupoid $\cG$ as follows.   Given a homotopy class $[\gamma]$ in
$\cG$, we choose a representative path $\gamma$ and choose a
lift $\tilde{\gamma}$ to $X$.  Since the path $\tilde{\gamma}$ starts
and ends at a vertex of $X$, we can find a minimum length cube path
$\bd$ in the same homotopy class as $\tilde{\gamma}$.  The $G$-weight
of this cube path is independent of all choices (by Proposition~\ref{prop:wwd}), so gives a well-defined
{\em $G$-weight} to the morphism $[\gamma]$.  We denote 
this weight by $\tau([\gamma])$. 

\begin{definition}  Let $G$ be a group acting freely, cellularly, and
  cocompactly on a connected CAT($0$) cube complex $X$.  Then the {\em
  multivariate growth series for $\cG$} is 
  the power series $\cG(\bt)\in\bbQ((\bt))$ defined by 
\[\cG(\bt)=\sum_{[\gamma]\in\cG}\tau([\gamma]).\]
The {\em (ordinary) growth series for $\cG$}
is the single-variable power series $\cG(t)\in\bbQ((t))$
obtained by substituting $t$ for each indeterminate $t_{h}$ in
the multivariate growth series.
\end{definition}

Given $x,y\in V/G$, we define the growth series $\cG_{x,y}(\bt)$ and
$\cG_{x,y}(t)$ by the same formulas, but only sum over homotopy paths
from $x$ to $y$.  In particular, each $x\in V/G$ determines growth
series for the group $G$, given by $G_x(\bt)=\cG_{x,x}(\bt)$ and
$G_x(t)=\cG_{x,x}(t)$.  It follows from the bijection
$\cL\rightarrow\cG$ and the discussions above, that the growth series
$\cG_{x,y}(\bt)$ is the specialization (discussed in \ref{ss:spec}) of
the characteristic series $\lambda_{x,y}$ with respect to the substitution  
$\cA\rightarrow\bbQ(\bt)$ that maps each diagonal $j=\bar{d}$ in $\cA$
to its corresponding weight $\tau([j])=\tau(d)$.  Applying this same
substitution to (\ref{eq:lambda+}) and (\ref{eq:lambda-}), gives the
following. 

\begin{proposition}\label{prop:pm-growth}
Let $Q_+(\bt)$ (resp., $Q_+(t)$) denote the transition matrix $Q_+$
with each letter $j=\bar{d}\in\cA$ replaced by the monomial
$\tau(d)$ (resp., $t^{|d|}$) where $d$ is any lift of $\bar{d}$.
Define $Q_-(\bt)$ and $Q_-(t)$ similarly.  Then we have the following
rational function representations:
\[\cG_{x,y}(\bt)=B_{\alpha}(I-Q_+(\bt))^{-1}E=B_{\omega}(I-Q_-(\bt))^{-1}E,\]
and
\[\cG_{x,y}(t)=B_{\alpha}(I-Q_+(t))^{-1}E=B_{\omega}(I-Q_-(t))^{-1}E.\]
\end{proposition}

\subsection{Examples}
For a first example, let $Y$ be the target graph in Figure~\ref{fig:ex1}, and
let $X$ be the universal cover.  Then $X$ is a connected CAT($0$) cube
complex and the group $G=\pi_1(Y)\cong\bbZ$ acts freely on $X$.  
\begin{figure}[ht]
\begin{center}
\psfrag{a}{$a$}
\psfrag{b}{$b$}
\psfrag{x}{$x$}
\psfrag{y}{$y$}
\psfrag{ta}{$\tilde{a}$}
\psfrag{tb}{$\tilde{b}$}
\psfrag{tx}{$\tilde{x}$}
\psfrag{ty}{$\tilde{y}$}
\includegraphics[scale = .6]{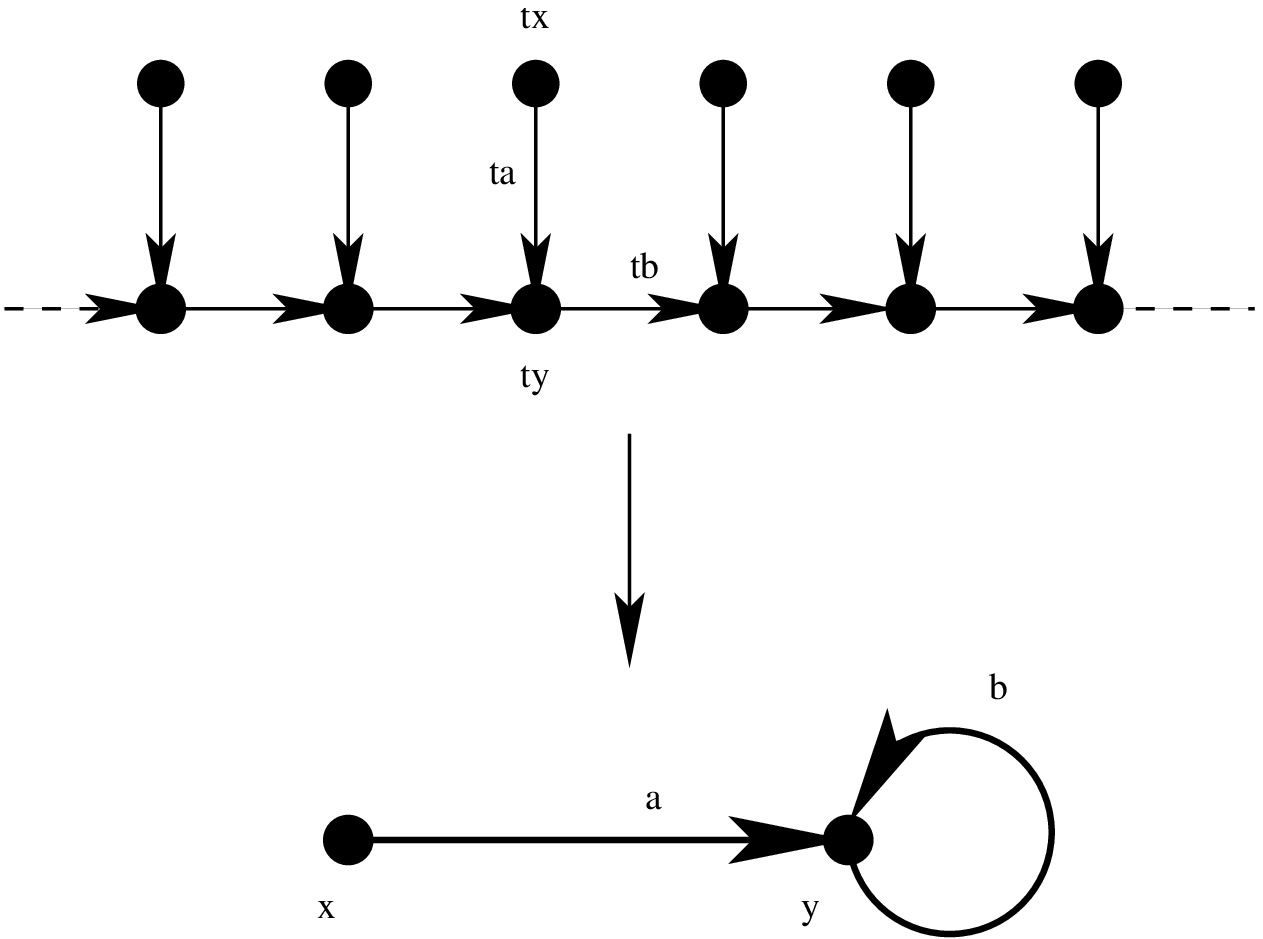}
\caption{\label{fig:ex1}}
\end{center}
\end{figure}
There are $2$ trivial diagonals $x$ and $y$ and $4$
nontrivial diagonals $a,a^*,b,b^*$ in $Y=X/G$.  Thus
$\cA=\{a,a^*,b,b^*\}$ and $\cS=\{a,a^*,b,b^*,x,y\}$.  The corresponding
transition matrix (with respect to the ordering $(x,y,a,a^*,b,b^*)$)
is given by 
\[Q_+=\left[\begin{array}{cccccc} 
0 & 0 & 0 & 0 & 0 & 0\\
0 & 0 & 0 & 0 & 0 & 0\\
0 & a & 0 & 0 & a & a\\
a^* & 0 & 0 & 0 & 0 & 0\\
0 & b & 0 & b& b & 0\\
0 & b^* & 0 &b^*& 0 & b^*\end{array}\right].\]
Substituting $t_1,t_2,t_3,t_4$ for the symbols $a,a^*,b,b^*\in\cA$,
respectively, and using Proposition~\ref{prop:pm-growth}, we 
obtain the multivariate growth series:
\[\cG_{x,x}(\bt) = 1-\frac{t_1t_2(2t_3t_4-t_3-t_4)}{(1-t_3)(1-t_4)},\hspace{.7in}
\cG_{x,y}(\bt) = \frac{t_1(1-t_3t_4)}{(1-t_3)(1-t_4)},\]
\[\cG_{y,x}(\bt) = \frac{t_2(1-t_3t_4)}{(1-t_3)(1-t_4)},\hspace{.7in}
\cG_{y,y}(\bt) = \frac{1-t_3t_4}{(1-t_3)(1-t_4)}.\]
With the substitutions $t_i=t$, we obtain the ordinary growth series:
\[\cG_{x,x}(t) = \frac{2t^3-t+1}{1-t} = 1+2t^3+2t^4+2t^5+\cdots.\]
\[\cG_{x,y}(t) =\cG_{y,x}(t)= \frac{t^2+t}{1-t} = t+2t^2+2t^3+2t^4+\cdots.\]
\[\cG_{y,y}(t) = \frac{1+t}{1-t}= 1+2t+2t^2+2t^3+\cdots.\]
   
For a more intricate example, and one that illustrates our reciprocity
formula, we let $X$ be the cell decomposition of the hyperbolic plane
$\bbH^2$ that is dual to the tesselation by right-angled hexagons
(Figure~\ref{fig:hextiling}).   
\begin{figure}
\begin{center}
\includegraphics[scale = .6]{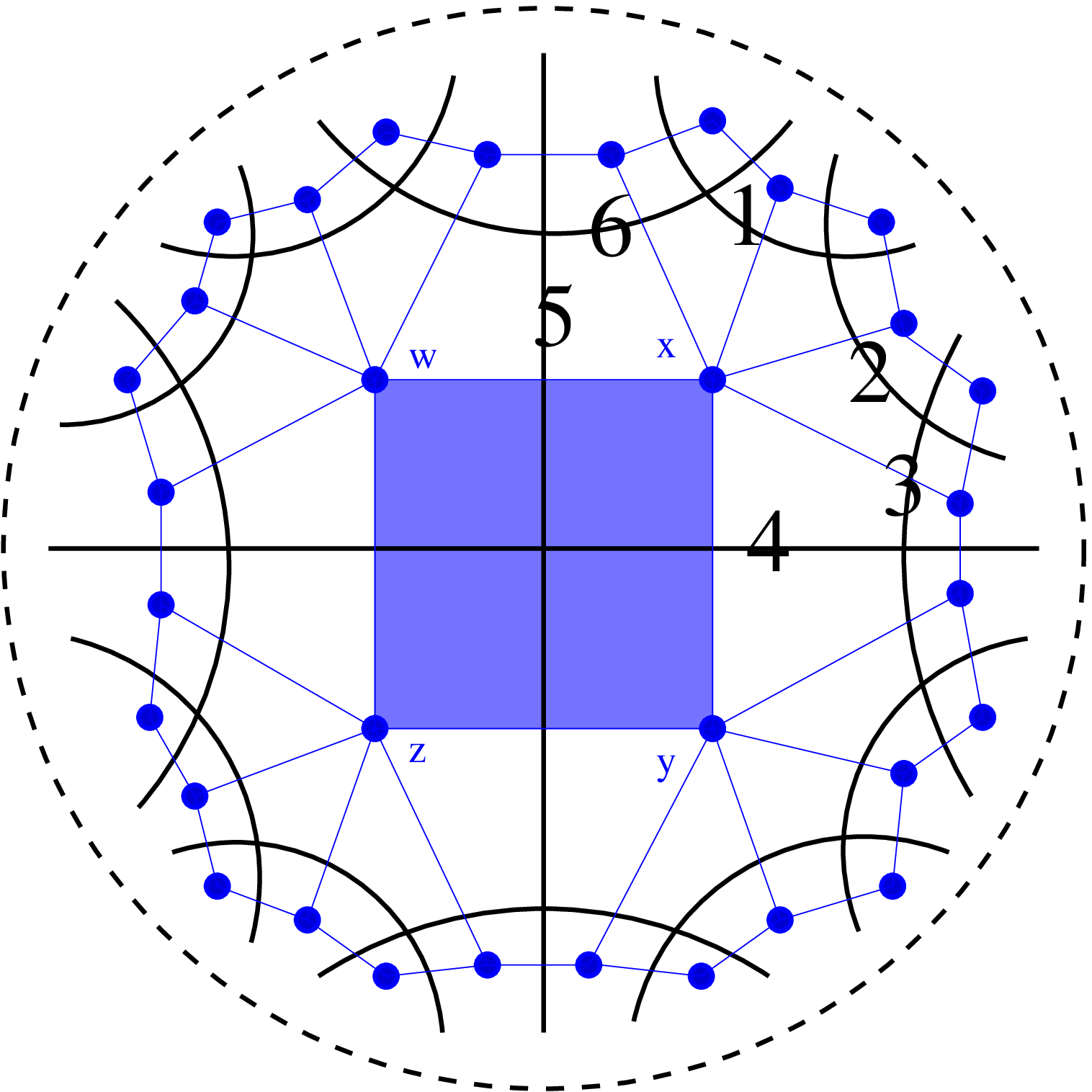}
\caption{\label{fig:hextiling}}
\end{center}
\end{figure}
All of the $2$-cells of $X$ are squares and if one gives each of them the
standard Euclidean metric, $X$ is a CAT($0$) square complex.
If $W$ is the right-angled Coxeter group generated by reflections
across the sides of one of the hexagons, then $W$ acts cellularly
and cocompactly (but not freely) on $X$.  If $s_1,\ldots,s_6$ are
the reflections across the lines numbered $1,2,3,4,5,6$ in the figure,
then there is a surjective homomorphism $W\rightarrow\bbZ_2\times\bbZ_2$
defined by mapping $s_1,s_3,s_5$ to $(1,0)$ and $s_2,s_4,s_6$ to
$(0,1)$.   The kernel $G$ of this homomorphism is an index $4$
subgroup acting freely on $X$, and the quotient $X/G$ is the surface
of genus $2$ shown in Figure~\ref{fig:ex2}.  
\begin{figure}
\begin{center}
\includegraphics[scale = .6]{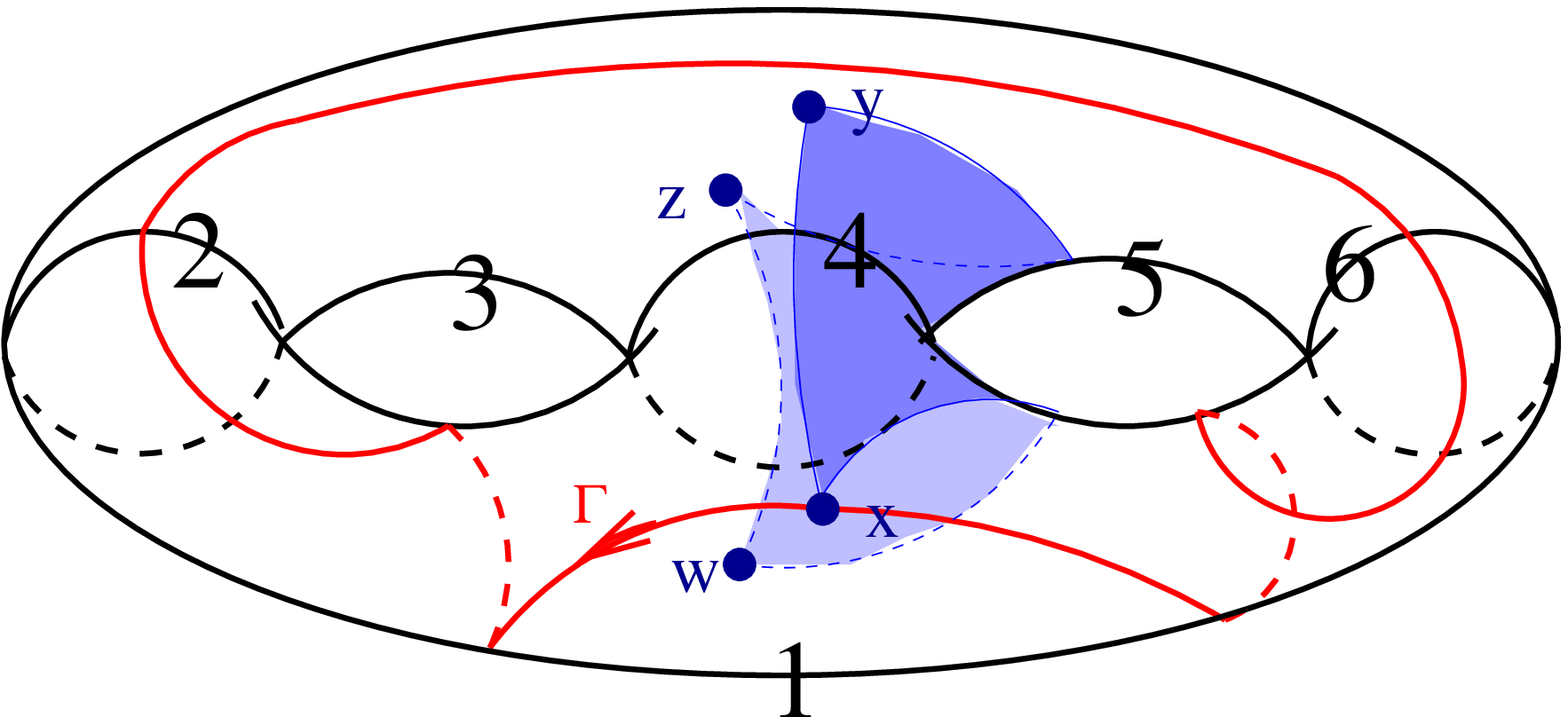}
\caption{\label{fig:ex2}}
\end{center}
\end{figure}

The cell structure on $X/G$ consists of $4$ vertices ($x,y,z,w$), $12$
$1$-cells, and $6$ $2$-cells.  It follows that the state set $\cS$ will have $52$
diagonals, $4$ of which are trivial.  There are $12$ orbits of edge
classes in $\cH/G$, so the matrix $I-Q_+(\bt)$ is a $52\times 52$ matrix
with $12$ indeterminates.  Inverting the single variable matrix
$I-Q_+(t)$  matrix is much more manageable and, with
Proposition~\ref{prop:pm-growth}, yields:
\[\cG_{x,x}(t) = \frac{1-2t^2+t^4}{1-14t^2+t^4},\;\;\; \cG_{x,y}(t)=
\frac{3t+3t^3}{1-14t^2+t^4},\;\;\;\cG_{x,z}(t) =\frac{12t^2}{1-14t^2+t^4}.\]
(By symmetry, each of the remaining series $\cG_{i,j}(t)$ is equal to
one of these.)  Note that the weighting on edge paths corresponding to
the single variable growth series has the interpretation of counting
the number of times the path crosses one of the $6$ numbered curves
(these are the hyperplane orbits) in Figure~\ref{fig:ex2}.  In
particular, expanding the series for 
$\cG_{x,x}=\pi_1(X/G,x)$ gives the growth series for the surface group $G$
relative to the basepoint $x$: 
\[G_x(t)=1+12t^2+168t^4+2340t^6+\cdots\]
so there are $12$ elements that cross $2$ curves, $168$ that cross $4$
curves, $2340$ that cross $6$ curves, and so on.  The indicated loop
$\Gamma$, for example, is a minimal length representative in its
homotopy class.  It crosses the curves numbered $2,3,5,6$ each once and crosses
the curve numbered $1$ twice, thus it represents one of the $2340$
group elements of length $6$.  

For this example, the surface $X/G$ is an Eulerian cube complex (in
fact, a manifold), so the hypotheses of our main theorem hold.  In
particular, all three of the growth series $\cG_{x,x}(t)$,
$\cG_{x,y}(t)$, and $\cG_{x,z}(t)$ satisfy the reciprocity formula
$G(t^{-1})=G(t)$.   

\section{Reciprocity for Eulerian cube complexes}\label{s:pf}

In this section we shall prove the main theorem stated in the
introduction, but in the more  general setting of characteristic series rather than growth series.
To state this theorem, we first extend the involution
$\bar{d}\mapsto\bar{d}^*$ defined on $\cA$ to the entire monoid
$\cA^*$.  Given a word $\bar{\bd}=\bar{d}_1\bar{d}_2\cdots\bar{d}_n$,
we define $\bar{\bd}^*$ to be
$\bar{d}_1^*\bar{d}_2^*\cdots\bar{d}_n^*$. By linearity, this extends
to an involution on $\bbQ((\cA))$.  Given a matrix $M$ defined over
$\bbQ((\cA))$ we define $M^*$ to be the matrix defined by
$(M^*)_{i,j}=M_{i,j}^*$.  It follows that if a series
$\lambda\in\bbQ((\cA))$ has rational expression of the form  
\[\lambda=B(I-Q)^{-1}E\]
as in \ref{ss:char-series}, then so does $\lambda^*$ and we have 
\begin{equation}\label{eq:lambda*}
\lambda^*=B(I-Q^*)^{-1}E.
\end{equation}

In terms of the characteristic series for normal cube paths, our
reciprocity formula is the following.

\begin{theorem}\label{thm:main}  Let $G$ be a group acting freely,
  cellularly, and cocompactly on a CAT($0$) cube complex $X$, and let
  $V$ be the set of vertices of $X$.  If $X/G$ is Eulerian of
  dimension $n$, then for any $x,y\in V/G$, the reciprocal
  $\bar{\lambda}_{x,y}$ exists and is given by 
\[\bar{\lambda}_{x,y}=(-1)^n\lambda_{x,y}^*.\]
\end{theorem}

To obtain the theorem stated in the introduction, we simply
specialize the series $\lambda$ and $\lambda^*$ to the commutative ring
$\bbQ((\bt))$.  Any diagonal $\bar{d}\in\cA$ is mapped to the product
of indeterminates $t_h$ where $h$ is a hyperplane in $X/G$ crossed by
$\bar{d}$.  Since $\bar{d}$ and $\bar{d}^*$ cross the same hyperplanes, 
both $\lambda_{x,y}$ and $\lambda^*_{x,y}$ specialize to the same
growth series $\cG_{x,y}(\bt)$.  Since $\bar{\lambda}_{x,y}$
specializes to $\cG_{x,y}(\bt^{-1})$, we then obtain the formula
\[\cG_{x,y}(\bt^{-1})=\cG_{x,y}(\bt)\]
from the introduction.  The remainder of this section will be
devoted to the proof of Theorem~\ref{thm:main}.

\subsection{Properties of the transition matrices} 
Let $Q$ denote the transition matrix $Q_+$ for the automaton given in
\ref{ss:gcs} for the language $\cL$.  Since the nonzero entries in
a given row are always equal, we can factor out a diagonal matrix
on the left.  We define $\cS\times\cS$ matrices $J_0$
and $D_0$ by
\[(J_0)_{i,j}=\left\{\begin{array}{ll}
(-1)^{|i|-1} & \mbox{if $Q_{i,j}\neq 0$}\\
0 & \mbox{otherwise}\end{array}\right.\]
and 
\[(D_0)_{i,j}=\left\{\begin{array}{ll}
(-1)^{|i|-1}i & \mbox{if $i=j\in\cA$}\\
0 & \mbox{otherwise}\end{array}\right.\]
where $|i|$ denotes the length $|d|$ for any of any lift $d$ for $i$.
Then we have the factorization $Q=D_0J_0$.  

To enable various matrix inversions, we define extensions of the matrices $J_0$ and
$D_0$ by replacing the zero diagonal entries corresponding to states in
$V/G$ with $\pm 1$.  We define $J$ and $D$ as follows:
\begin{equation}\label{eq:J}
J_{i,j}=\left\{\begin{array}{ll}
-1 & \mbox{if $i=j\in V/G$}\\
(J_0)_{i,j} & \mbox{otherwise}\end{array}\right.
\end{equation}
and 
\[D_{i,j}=\left\{\begin{array}{ll}
1 & \mbox{if $i=j\in V/G$}\\
(D_0)_{i,j} & \mbox{otherwise}\end{array}\right.\]
We now collect various properties of these matrices
for future reference.  Recall that the reduced Euler characteristic of
a simplicial complex $K$, denoted by $\bar{\chi}(K)$, is $\chi(K)-1$
where $\chi(K)$ is the usual Euler characteristic.   

\begin{lemma}\hspace{.1in} \label{lem:main}  
\begin{enumerate}
\item $Q=D_0J_0=DJ_0=D_0J$.
\item $[*]D[*]=D^*$ and $[*]D_0[*]=D_0^*$. 
\item For any $j\in\cS$, the sum of the entries in the $j$th column of $[*]J$ is
  equal to $\bar{\chi}(\link(\alpha(j)))$.  
\item If $X/G$ is Eulerian, then $J$ is invertible and $J^{-1}=[*]J[*]$.
\end{enumerate}
\end{lemma}

\begin{proof}
The first two properties follow directly from the definitions.  For
the second two, we notice that the entries of the matrix $[*]J$ are
given by 
\[([*]J)_{i,j}=J_{i^*,j}=\left\{\begin{array}{ll}
-1 & \mbox{if $i=j\in V/G$}\\
(-1)^{|i|-1} & \mbox{if $i\in\cA, j\in V/G$ and $\alpha(i)=j$}\\
(-1)^{|i|-1} & \mbox{if $i,j\in\cA$, $\alpha(i)=\alpha(j)$, and
  $\star(\sigma(i))\cap\sigma(j)=\emptyset$}\\
0 & \mbox{otherwise.}\end{array}\right.\]
In particular, the $(i,j)$-entry is zero unless
$\alpha(i)=\alpha(j)$, so the matrix breaks into diagonal blocks
corresponding to the partition of $\cS=\coprod_x\cS_x$ where $\cS_x$ is the set
of diagonals with initial vertex $x$.  That is 
\[[*]J=\bigoplus_{x\in V/G} J_x\]
where $J_x$ is the $\cS_x\times\cS_x$ block of $[*]J$.  For each $x$, the
matrix $J_x$ is the ``anti-incidence matrix'' introduced in
Definition~7.1 of \cite{Scott} for the simplicial complex $\link(x)$.  By
Theorem~7.13 of \cite{Scott}, the columns of $J_x$ all have sum
$\bar{\chi}(\link(x))$.  By the block decomposition of $[*]J$, the
$j$th column of $[*]J$ has sum $\bar{\chi}(\link(\alpha(j)))$, proving (3). 
The second part of Theorem~7.13 of \cite{Scott} states that the
anti-incidence matrix for $\link(x)$ is an involution if $\link(x)$ is
an Eulerian sphere.  It follows that $J_x$, and hence $[*]J$ is an
involution, proving (4).
\end{proof}

\subsection{Proof of Reciprocity}
We now prove Theorem~\ref{thm:main}.  As above, we let $Q=Q_+$. By
definition, the reciprocal $\bar{\lambda}_{x,y}$ exists if the matrix
$I-\bar{Q}$ (obtained by 
replacing every entry $j\in\cA$ with $j^{-1}$) is invertible over
$\bbQ((\cA))$, in which case we have 
\begin{eqnarray}\label{eq:bl}
\bar{\lambda}_{x,y} &=& B_{\alpha}(I-\bar{Q})^{-1}E.
\end{eqnarray}
Since $D^{-1}=\bar{D}$, we can rewrite $I-\bar{Q}$ as:  
\begin{eqnarray*}
I-\bar{Q} &=& I-\bar{D}_0J\\
            &=& \bar{D}D-\bar{D}_0J\\
            &=& \bar{D}D_0-\bar{D}J,
\end{eqnarray*}
where the last equation follows from the fact that $\bar{D}D$ is
obtained from $\bar{D}D_0$ by putting $1$s in entries $i,j$ whenever
$i=j\in V/G$, and $\bar{D}J$ is obtained from $\bar{D}_0J$ by putting
$-1$s in these same entries.  Now using algebra and (1), (2), and (4) of
Lemma~\ref{lem:main}, we have 
\begin{eqnarray*}
I-\bar{Q}   &=& -\bar{D}(I-D_0J^{-1})J\\
            &=& -\bar{D}(I-D_0[*]J[*])J\\
            &=& -\bar{D}(I-[*]D^*_0J[*])J\\
            &=& -\bar{D}(I-[*]Q^*[*])J.
\end{eqnarray*}
By Proposition~\ref{prop:pmQ}, this can be rewritten as 
\[I-\bar{Q} = -\bar{D}(I-Q_-^*)J,\]
which is invertible with inverse given by 
\[(I-\bar{Q})^{-1} = -[*]J[*](I-Q_-^*)^{-1}D.\]
Substituting into (\ref{eq:bl}), we have  
\begin{eqnarray}\label{eq:euler}                    
\bar{\lambda}_{x,y} &=& -B_{\alpha}[*]J[*](I-Q_-^*)^{-1}DE.
\end{eqnarray}
Since $D$ is diagonal with $(y,y)$-entry equal to $1$, we have $DE=E$.  
The vector $B_{\alpha}$ only has nonzero entries (which are all $1$s)
for those states $j$ whose initial vertex is $x$, hence by the block
decomposition of $[*]J$ and part (3) of Lemma~\ref{lem:main}, we have 
$B_{\alpha}[*]J=\bar{\chi}(\link(x))B_{\alpha}$.  Substituting into
(\ref{eq:euler}), and using the fact that $B_{\alpha}[*]=B_{\omega}$, gives   
\begin{eqnarray*}
\bar{\lambda}_{x,y} &=& -\bar{\chi}(\link(x))B_{\omega}(I-Q_-^*)^{-1}E.
\end{eqnarray*}
Since the vertex links of $X/G$ are Eulerian spheres of
dimension $n-1$, this simplifies to 
\begin{eqnarray*}
\bar{\lambda}_{x,y} &=& (-1)^nB_{\omega}(I-Q_-^*)^{-1}E.\\
\end{eqnarray*}
Finally, using the rational expression for $\lambda_{x,y}$ in
(\ref{eq:lambda-}) and equation (\ref{eq:lambda*}), we then have 
\[\bar{\lambda}_{x,y}=(-1)^n\lambda_{x,y}^*.\]
This completes the proof.

\subsection{Generalization to orbihedra}\label{ss:orb}

One can weaken the assumption slightly on the $G$-action and assume
only that the action on the vertex set $V$ is free.  Since one then still
has unique lifts from vertex links in $X/G$ to vertex links in 
$X$, all of the above arguments go through verbatim.  In particular,
the main result of this paper then includes the case proved by the author in
\cite{Scott} where $G$ is a right-angled Coxeter group and $X$ is the
corresponding Davis complex.  In this case, there is a single free
vertex orbit, but every (nonoriented) diagonal in $X$ has an order-$2$
stabilizer that reverses its direction. It follows that the 
involution $*$ on the state set $\cS$ is trivial (so $J$ is
its own inverse) and the action of $\ast$ on $\bbQ((\bt))$ is trivial.
The reciprocity formula therefore takes the form 
\[\bar{\lambda}_{x,y}=(-1)^n\lambda_{x,y}.\]   

Theorem~\ref{thm:main} also applies more generally to the
right-angled {\em mock reflection groups} studied in
\cite{DJS2,Scott2}.  By definition, $G$ is a right-angled mock
reflection group if it acts on a CAT($0$) cube complex $X$ with the
property that the action is simply-transitive on the vertex set and
the edge stabilizers are all nontrivial.  (Any right-angled Coxeter group
acting on its Davis complex is a special case.)  Again one has a
single free vertex orbit and for each edge, an involution in the group
that reverses its orientation.  It follows that the involution
$*:\cS\rightarrow\cS$ on the state set is trivial on the $0$ and
$1$-dimensional diagonals in $X/G$.  However, for mock reflection
groups, the hyperplane dual to 
an edge need not be stabilized pointwise by the edge stabilizer, so
the involution $*$ need not be trivial on higher dimensional
diagonals.  This time, the reciprocity formula for the characteristic
series requires the more general form
$\bar{\lambda}_{x,y}=(-1)^n\lambda_{x,y}^{\ast}$.
However, since $\lambda_{x,y}$ and $\lambda_{x,y}^*$ specialize to the
same commutative-variable series, one still obtains 
\[\cG_{x,y}(\bt^{-1})=(-1)^n\cG_{x,y}(\bt).\]

In light of this generalization to orbihedra, our theorem applies to
finite index subgroups of right-angled Coxeter groups and, more
generally, to finite index subgroups of right-angled mock reflection
groups. 

\begin{corollary}
Let $W$ be a right-angled Coxeter group or mock reflection group
acting on its corresponding CAT($0$) cube complex $X$, and assume that
$X$ is Eulerian of dimension $n$. If $G$ is any finite index subgroup
of $W$, and $x$ is the $G$-orbit of any vertex, then the growth series
of $G$ with respect to $x$ satisfies     
\[G_x(\bt^{-1})=(-1)^n G_x(\bt).\]
\end{corollary}

\bibliographystyle{amsplain}
\bibliography{CCR}

\end{document}